\documentclass[11pt]{article}%
\usepackage{amsfonts}
\usepackage{amsmath}
\usepackage{amssymb}
\usepackage{euler}
\usepackage{graphicx}%
\providecommand{\U}[1]{\protect\rule{.1in}{.1in}}
\newtheorem{theorem}{Theorem}

\newtheorem{corollary}[theorem]{Corollary}

\newtheorem{proposition}[theorem]{Proposition}

\newenvironment{proof}[1][Proof]{\noindent\textbf{#1.} }{\ \rule{0.5em}{0.5em}}
\begin{document}

\title{\textsf{A Constructive Version of Ekeland's Variational Principle}}
\author{\textsf{Douglas S Bridges}}
\maketitle

\begin{abstract}%
\noindent
\textsf{Building on preliminary results about lower sections of real-valued
functions on a metric space, we provide a constructive counterpart of
Ekeland's theorem on the approximate optimisation of real-valued functions.}

\end{abstract}

%

\normalfont\sf

\section{Introduction}

The framework of this article\footnote{%
\normalfont\sf
MSC Classifications: 03F60, 26E40
\par
\ \ \ Keywords: constructive, Ekeland} is Bishop-style constructive analysis
(\textbf{BISH} \cite{Bish, BB,Handbook}), which, in practice, is mathematics
carried out with intuitionistic logic, an appropriate constructive set theory
(see \cite[Ch. 2]{Handbook} or \cite{Aczel,dsbMorse}) or type theory
\cite{ML}, and dependent choice. The remarkable fact is that everything proved
within that framework has a fully computational interpretation within
recursion theory (\cite{Aberth, Kushner},\cite[Ch. 3]{BR}), Brouwer's
intuitionistic mathematics (\cite{Dummett,TvD}, \cite[Ch. 5]{BR}), Weihrauch's
Type II Effectivity theory \cite{Bauer,Wei}, and, we believe, any good model
of computable mathematics. It is also the case that everything in Bishop's
constructive mathematics holds in classical (that is, traditional) mathematics.

We shall be primarily interested in a constructive version of the classical
\textsc{Ekeland Variational Principle (EVP)}:

\begin{quote}
Let $\left(  X,\rho\right)  $ be a complete metric space and $f:X\rightarrow
\overline{\mathbb{R}}$ a lower semicontinuous function that is bounded below.
Let $x_{0}\in X$, $f(x_{0})\in\mathbb{R}$, and $c>0$. Then there exists $v\in
X$ such that $f(v)+c\rho(x_{0},v)\leqslant f(x_{0})\ $and $f(v)<f(x)+c\rho
(x,v)\ $for each $x\neq x_{0}$.
\end{quote}

%

\noindent
Here, $\overline{\mathbb{R}}\mathbb{\ }$denotes the extended real line. The
standard classical proof of EVP, as found on pages 15--16 of \cite{Aubin}%
,\footnote{%
\normalfont\sf
or https://en.wikipedia.org/wiki/Ekeland\%27s\_variational\_principle} fails
constructively because, in contrast to the classical situation, being bounded
below is not a sufficient condition for the infimum of a subset $S$ of
$\mathbb{R}$ to exist in \textbf{BISH}; but if $S$ is compact, then
constructively it does have an infimum.\footnote{%
\normalfont\sf
It has been known since the time of Brouwer that although a uniformly
continuous function on a totally bounded metric space has an infimum, it
cannot be proved constructively that if the space is compact, then the infimum
is attained (is a minimum). It seems hard to produce good conditions that
ensure the attainment of such an infimum, at least without invoking some
additional principle such as a version of Brouwer's fan theorem; see
\cite{bbs,dsbopt,Schuster}.} In fact, our proof of Theorem \ref{Ekeland} below
will need $X$ to be locally compact.

But infima are not the only hurdle we need to jump over for a constructive
proof. In \textbf{BISH} we cannot be sure that if $X$ is locally compact and
$f$ is both bounded below and uniformly continuous, then sets of the form
$\left\{  x\in X:f(x)\leqslant\lambda\right\}  $ are totally bounded for all
$\lambda>\inf f$. But we can prove that for all but countably many real values
$\lambda>\inf f$, the set in question is totally bounded (Proposition
\ref{Bisha} below).

Our constructive version of EVP (Theorem \ref{Ekeland} below) replaces the
completeness of $X$ by local compactness and requires continuity of the
function $f:X\rightarrow\mathbb{R}$. The paper ends with a brief indication of
where the theorem fits into the intuitionistic and recursive varieties of
constructive mathematics. Altogether, the work below can be viewed as a
contribution to constructive optimisation theory, complementing and, we hope,
eventually extending the scope of that subject (see \cite[Chs. 11 and
12]{Handbook}, \cite{dsbopt,Schuster}).

We assume a minimal technical background in constructive mathematics, as in
the relevant parts of Chapters 1--4 of \cite{bv} or Ishihara's Chapter 8
in\footnote{%
\normalfont\sf
The Handbook \cite{Handbook}\ is the most up-to-date reference in constructive
mathematics.} \cite{Handbook}; other background references include
\cite{Bish,BB,BR,TvD}. However, before delving into our main topic, it is
convenient, and potentially helpful, to put on record here some constructive
analysis dealing with a metric space\footnote{%
\normalfont\sf
We use $\rho$ to denote the metric on any metric space.} $(X,\rho)$. First,
the \textsc{inequality} $\neq$ on $X$ is defined by%
\[
x\neq y\Leftrightarrow\rho(x,y)>0.
\]
and the \textsc{complement}\textbf{\ }of a subset $S$ of $X$ is the set%
\[
X%
\mathord{\sim}%
S\equiv\left\{  x\in X:\forall_{s\in S}\left(  x\neq s\right)  \right\}  .
\]
We define $S\subset X$ to be

\begin{itemize}
\item[$\blacktriangleright$] \textsc{inhabited} if we can construct an object
belonging to it;

\item[$\blacktriangleright$] \textsc{located} if $\rho(x,S)\equiv\inf\left\{
\rho(x,s):s\in S\right\}  \ $exists for each $x$ in $X$;

\item[$\blacktriangleright$] \textsc{totally bounded}\textbf{\ }if for each
$\varepsilon>0$ there exists an inhabited finite subset---called
an\textbf{\ }$\varepsilon$\textbf{-approximation }to $S$---of $S$ with the
property: for each $x\in S$ there exists $k\leqslant n$ with $\rho
(x,x_{k})<\varepsilon$;

\item[$\blacktriangleright$] \textsc{compact}\textbf{\ }if it is totally
bounded and complete;

\item[$\blacktriangleright$] \textsc{locally compact} if it is inhabited and
every bounded subset is contained in a compact subset of $X$.
\end{itemize}

%

\noindent
Note that we require sets that are totally bounded, compact, or inhabited to
be inhabited. If $S$ is locally compact, then it is complete and located;
moreover, by Proposition 2.2.18 of \cite{bv}, every located closed subspace of
$S$ is locally compact.

We say that a subset $S$ of $X$ is a \textsc{compact image}\textbf{\ }if there
is a uniformly continuous mapping of a compact metric space \emph{onto}
$S$.\footnote{%
\normalfont\sf
Classically, this would mean that $S$ itself is compact; but in the recursive
variety of constructive analysis there exists a uniformly continuous mapping
of the compact interval $\left[  0,1\right]  $ onto the half-open interval
$(0,1]$ (\cite{JR1984},\cite[pp. 134--136]{BR}), so we cannot expect to prove
that compact images are themselves compact.} A mapping $f$ of $X$ into a
metric space $Y$ is \textsc{continuous} if it is \textsc{uniformly continuous
near each compact image} $K$ of $X$: that is, for each $\varepsilon>0\ $there
exists $\delta>0$ such that if $x\in K,\,x^{\prime}\in X,$ and $\rho
(x,x^{\prime})<\delta$, then $\rho(f(x),f(x^{\prime}))<\varepsilon$. With this
notion (originating with Bishop \cite{Bishneat}) we can prove that sums,
products, etc., and---significantly, because mere uniform continuity on
compact subsets does not cover this in \textbf{BISH}---compositions of
continuous functions are continuous. Because singleton subsets of a metric
space are compact, continuity implies pointwise continuity. Hence if $f$ is
continuous, then it is \textsc{strongly extensional}: that is, $x\neq
x^{\prime}$ whenever $x,x^{\prime}\in X$ and $f(x)\neq f(x^{\prime})$. If $X$
is locally compact, then the following conditions are equivalent:

\begin{itemize}
\item[$\bullet$] $f$ is continuous (uniformly near each compact subset of
$X).$

\item[$\bullet$] $f$ is uniformly continuous on each compact subset of $X;$

\item[$\bullet$] $f$ is uniformly continuous on each bounded subset of $X$.
\end{itemize}

%

\noindent
In that case, $f$ maps bounded sets onto bounded sets.

\section{Sections and the Constructive EVP}

For a mapping $f$ of $X$ into $\mathbb{R}$ and for each $\lambda\in\mathbb{R}$
we define the corresponding \textsc{lower section}%
\[
\mathsf{L}\left(  f,\lambda\right)  =\left\{  x\in X:f(x)\leqslant
\lambda\right\}  .
\]
and \textsc{epigraph}%
\[
\mathsf{epi}(f)=\left\{  \left(  x,\lambda\right)  :x\in X,\,\lambda
\in\mathbb{R}\mathbf{,\,}f(x)\leqslant\lambda\right\}  .
\]
%

\noindent
Our first result needs only sequential continuity.

\begin{proposition}
\label{may12p1}Let $X\,$\ be a metric space and $f$ a sequentially continuous
mapping of $X$ into $\mathbb{R}$. Then

\begin{enumerate}
\item[\emph{(i)}] $\mathsf{epi}(f)$ is closed in $X$, and

\item[\emph{(ii)}] every lower section of $f$ is closed in $X$.
\end{enumerate}
\end{proposition}

\begin{proof}
Let $\left(  \left(  x_{n},\lambda_{n}\right)  \right)  _{n\geqslant1}$ be a
sequence in $\mathsf{epi}(f)$ converging to a limit $\left(  x,\lambda\right)
$ in $X$. Suppose that $f(x)>\lambda$ and let $\alpha=\frac{1}{2}%
(\lambda+f(x))$. Then $f(x)>\alpha>\lambda$. By the sequential continuity of
$f$ at $x$, for all sufficiently large $n$ we have $f(x_{n})>\alpha$. But
$\lambda=\lim_{n\rightarrow\infty}\lambda_{n}$, so for all sufficiently large
$n$ we have $\lambda_{n}<\alpha$ and therefore $f(x_{n})>\lambda_{n}$. This is
absurd, since $\left(  x_{n},\lambda_{n}\right)  \in\mathsf{epi}(f)$. Hence
$f(x)\leqslant\lambda$ and therefore $\left(  x,\lambda\right)  \in
\mathsf{epi}(f)$. An analogous argument proves (ii).
\end{proof}

%

\medskip

A property $P(\lambda)$ \textsc{holds for all but countably many} $\lambda$ in
$\mathbb{R}$ if there exists a sequence $\left(  \lambda_{m}\right)
_{m\geqslant1}$ of real numbers such that $P(\lambda)$ holds for all
$\lambda\in\mathbb{R}$ with $\lambda\neq\lambda_{m}$ for each $m$. We now
state without proof a fundamental result about compact metric spaces.

\begin{proposition}
\label{Bisha}Let $X$ be a compact metric space, and $f:X\rightarrow\mathbb{R}$
a continuous function. Then $\mathsf{L}\left(  f,\lambda\right)  $ is compact
for all but countably many real numbers $\lambda>\inf f$ \emph{\cite[Ch. 4,
(4.9)]{BB}}.
\end{proposition}

%

\medskip

We denote the open (respectively, closed) ball with centre $a$ and radius
$r>0$ in a metric space $X$ by $B_{X}(a,r)$ (respectively, $\overline{B}%
_{X}(a,r)$).

\begin{corollary}
\label{may26c1}If $X$ is a locally compact metric space and $a\in X$, then
$\overline{B}_{X}(a,r)$ is compact for all but countably many $r>0$.
\end{corollary}

\begin{proposition}
\label{may23p1}Let $f$ be a continuous mapping of a locally compact metric
space $X$ into $\mathbb{R}$, such that $\inf f$ exists in $\mathbb{R}$, and
let $a\in X$. Then $\mathsf{L}(f,\lambda)$ is locally compact for all but
countably many $\lambda>\inf f$.
\end{proposition}

\begin{proof}
Using Corollary \ref{may26c1}, choose a strictly increasing sequence $\left(
r_{n}\right)  _{n\geqslant1}$ of positive numbers diverging to $\infty$ such
that $B_{n}\equiv\overline{B}_{X}\left(  x_{0},r_{n}\right)  $ is compact. By
Proposition \ref{Bisha}, for each $n$ there exists a sequence $\left(
\lambda_{n,k}\right)  _{k\geqslant1}$ in $\mathbb{R}$ with each $\lambda
_{n,k}>\inf f(B_{n})\geqslant\inf f$, such that if $\lambda>\inf f$ and
$\lambda\neq\lambda_{n,k}$ for each $k$, then
\[
B_{n}\cap\mathsf{L}(f,\lambda)=\left\{  x\in B_{n}:f(x)\leqslant
\lambda\right\}
\]
is compact. Let $\left(  \lambda_{m}\right)  _{m\geqslant1}$ be an enumeration
of the real numbers $\lambda_{n,k}$ and $\inf f(B_{n})\ \left(  n,k\geqslant
1\right)  $. Let $\lambda>\inf f$ and $\lambda\neq\lambda_{m}$ for each $m$.
Given any bounded subset $B$ of $\mathsf{L}(f,\lambda)$, choose $N$ such that
$B\subset B_{N}\cap\mathsf{L}(f,\lambda)$. If $x\in B$, then $x\in B_{N}$ and
$\inf f(B_{N})\leqslant(x)\leqslant\lambda$, so $\lambda>\inf f(B_{N})$, and
$\lambda\neq\lambda_{N,k}$ for each $k\geqslant1$. Hence $B_{N}\cap
\mathsf{L}(f,\lambda)$ is a compact subset of $\mathsf{L}(f,\lambda)$ that
includes $B$.%

\hfill

\end{proof}

\begin{corollary}
\label{may25c3}Let $f$ be a continuous mapping of a locally compact metric
space $X$ into $\mathbb{R}$ such that $\inf f$ exists. Let $a\in X$ and $c>0$.
Then the infimum of the mapping $g:x\rightsquigarrow f(x)+c\rho(x,a)$ on $X$
exists, and $\mathsf{L}(g,\lambda)$ is compact for all but countably many
$\lambda>\inf g$.
\end{corollary}

\begin{proof}
If necessary replacing $f$ by $f+\inf f$, we may assume that $\inf
f\geqslant0$. Let $\alpha<\beta$, choose $r>\alpha$ such that $\overline
{B}_{X}(a,r)$ is compact, and set%
\[
\mu=\inf\left\{  g(x):x\in\overline{B}_{X}(a,r)\right\}  .
\]
Either $\mu>\alpha$ or $\mu<\beta$. In the first case, for all $x\in X$ either
$\rho(x,a)<r$ and therefore $g(x)\geqslant\mu>\alpha$, or else $\rho
(x,a)>\alpha$ and therefore $g(x)=f(x)+\rho(x,\alpha)>\alpha$. Thus if
$\mu>\alpha$, then $g(x)>\alpha$ for all $x\in X$. On the other hand, if
$\mu<\beta$, then there exists $x\in\overline{B}_{X}(a,r)$ with $g(x)<\beta$.
The constructive least-upper-bound principle \cite[Thm. 2.1.18]{bv} now shows
that $\inf g$ exists; it is clearly positive. Since $g$ is continuous, we see
from Proposition \ref{may23p1} that $\mathsf{L}(g,\lambda)$ is locally compact
for all but countably many $\lambda>\inf g$. But for such $\lambda$ and all
$x\in\mathsf{L}(g,\lambda)$ we have $c\rho(x,a)\leqslant g(x)\leqslant\lambda
$; so $\mathsf{L}(g,\lambda)\subset\overline{B}_{X}(a,c^{-1}\lambda)$,
$\mathsf{L}(g,\lambda)$ is bounded, and therefore it is compact.
\end{proof}

%

\medskip

If $X$ is a compact metric space, then its \textsc{diameter},%
\[
\mathsf{diam}(X)\equiv\sup\left\{  \rho(x,y):x,y\in X\right\}
\]
exists as the supremum of the uniformly continuous function $\rho$ on the
compact product space $X\times X$.

\begin{proposition}
\label{may04p1}Let $X$ be a metric space, and $\left(  K_{n}\right)
_{n\geqslant1}$ a sequence of compact subsets of $X$ that

\begin{itemize}
\item[\emph{(a)}] is decreasing---that is, $K_{n+1}\subset K_{n}$ for each
$n$--- and\footnote{%
\normalfont\sf
Dropping hypothesis (b) from Proposition \ref{may04p1}, and using the
classical---but not constructive---sequential compactness of the sets $K_{n}$,
we can easily prove that, simply under the hypothesis (a), $%
{\textstyle\bigcap_{n\geq1}}
K_{n}$ is inhabited. Here, however, is a Brouwerian example showing that we
cannot prove this constructively. Let $\left(  a_{n}\right)  _{n\geq1}$ be a
binary sequence with at most one term equal to $1$ (more formally, $a_{m}%
a_{n}=0$ whenever $m\neq n$). If $a_{k}=0$ for all $k\leq n$, set
$K_{n}=\left\{  0,1\right\}  $; if $a_{n}=1$, then set $K_{n}=\left\{
0\right\}  $ if $n$ is even, and $K_{n}=\{1\}$ if $n$ is odd. Then $\left(
K_{n}\right)  _{n\geq1}$ is a decreasing sequence of compact subsets of
$\left\{  0,1\right\}  $. If there exists $\xi$ in $%
{\textstyle\bigcap_{n\geq1}}
K_{n}$, then either $\xi>0$ or $\xi<1$. In the first case, $a_{n}=0$ for all
even $n$, and in the second, $a_{n}=0$ for all odd $n$. Thus the
aforementioned classical result implies the essentially nonconstructive
principle \textsf{LLPO: \ }If $\left(  a_{n}\right)  _{n\geq1}$ is a binary
sequence with at most one term equal to $1$, then either all even-indexed
terms $a_{n}$ equal $0$, or else all odd-indexed terms equal $0$.}

\item[\emph{(b)}] such that $\mathsf{diam}(K_{n})\rightarrow0$ as
$n\rightarrow\infty$.
\end{itemize}

%

\noindent
For each $n$ let $x_{n}\in K_{n}$. Then $\left(  x_{n}\right)  _{n\geqslant1}$
converges to the unique point $\xi$ in $%
{\textstyle\bigcap_{n\geqslant1}}
K_{n}$.
\end{proposition}

\begin{proof}
Given $\varepsilon>0$, choose $\nu$ such that $\mathsf{diam}(K_{\nu
})<\varepsilon$. By (a), if $m>n\geqslant\nu$, then $x_{m},x_{n}\in K_{N}$, so
$\rho(x_{m},x_{n})<\varepsilon$. Hence $\left(  x_{n}\right)  _{n\geqslant
1}\,\ $is a Cauchy sequence, and therefore converges to a limit $\xi$, in the
complete set $K_{1}$. Since for each $N$ the Cauchy sequence $\left(
x_{n}\right)  _{n\geqslant N}$ belongs to the complete set $K_{N}$, we see
that $\xi\in K_{N}$. Hence $\xi\in%
{\textstyle\bigcap_{n\geqslant1}}
K_{n}$. The uniqueness of $\xi$ follows from (b).%
\hfill

\end{proof}

%

\medskip

This brings us to our constructive EVP\textbf{.}

\begin{theorem}
\label{Ekeland}Let $f$ be a continuous real-valued mapping on a locally
compact metric space $X$, such that $\inf f$ exists. Let $x_{0}\in X$, $c>0,$
and $\varepsilon>0$. Then there exists $v\in X$ such that%
\begin{equation}
f(v)+c\rho\left(  v,x_{0}\right)  \leqslant f\left(  x_{0}\right)  \label{1}%
\end{equation}
and%
\begin{equation}
f\left(  v\right)  <f(x)+c\rho\left(  v,x\right)  +\varepsilon. \label{2}%
\end{equation}
for all $x\neq v$ in$\ X$.
\end{theorem}

\begin{proof}
Define the mapping $g_{0}:X\rightarrow\mathbb{R}$ by $g_{0}(x)\equiv
f(x)+c\rho(x,x_{0})$. Then $\inf g_{0}$ exists, by Corollary \ref{may25c3}. We
may assume that $\inf f=0$ and therefore $\inf g_{0}$ is positive. By
Corollary \ref{may25c3}, there exists $\varepsilon_{0}\ $with $0<\varepsilon
_{0}<\frac{1}{2}\min\left\{  \varepsilon,f(x_{0})\right\}  $ and
$f(x_{0})-\varepsilon_{0}\neq\inf g_{0}$, such that if $f(x_{0})-\varepsilon
_{0}>\inf g_{0}$, then $\mathsf{L}(g_{0},f(x_{0})-\varepsilon_{0})$ is
compact. If $f(x_{0})-\varepsilon_{0}<\inf g_{0}$, then $\mathsf{L}%
(g_{0},f(x_{0})-\varepsilon_{0})$ is empty, in which case for all $x\in X$,
\[
f(x_{0})\leqslant f(x)+c\rho(x,x_{0})+\varepsilon_{0}<f(x)+c\rho
(x,x_{0})+\varepsilon
\]
and $x_{0}$ itself satisfies the desired conditions for $v$. Thus we may
assume that $\mathsf{L}(g_{0},f(x_{0})-\varepsilon_{0})$\ is compact. Setting
$\lambda_{0}=0$ and $K_{0}=\mathsf{L}(g_{0},f(x_{0})-\varepsilon_{0})$, we
construct an increasing binary sequence $\left(  \lambda_{n}\right)
_{n\geqslant0}$; a sequence $\left(  g_{n}\right)  _{n\geqslant0} $ of
continuous real-valued functions on $X$; a sequence $\left(  K_{n}\right)
_{n\geqslant0}$ of compact subsets of $X$; a sequence $\left(  \varepsilon
_{n}\right)  _{n\geqslant0}$ of positive numbers; and a sequence $\left(
x_{n}\right)  _{n\geqslant0}$ of points of $X$, such that the following hold
for each applicable $n$.

\begin{itemize}
\item[(a)] $0<\varepsilon_{n}<\frac{1}{2}\min\left\{  \varepsilon
,f(x_{n})\right\}  $.

\item[(b)] If $\lambda_{n}=0$, then $g_{n}(x)=f(x)+c\rho(x,x_{n})$,
$f(x_{n})-\varepsilon_{n}>\inf g_{n}$,$\ K_{n}\equiv\mathsf{L}(g_{n}%
,f(x_{n})-\varepsilon_{n})$ is compact, $x_{n+1}\in K_{n}\,$, and
$f(x_{n+1})<\inf f(K_{n})+2^{-n-1}$.

\item[(c)] If $\lambda_{n}=1-\lambda_{n-1}$, then $g_{n}(x)=f(x)+c\rho
(x,x_{n})$, $f(x_{n})-\varepsilon_{n}<\inf g_{n}$, and for all $m\geqslant n $
we have $\lambda_{m}=1$, $g_{m}=g_{n}$, $K_{m}=\left\{  x_{n}\right\}  $, and
$x_{m}=x_{n}$.

\item[(d)] If $n\geqslant1$, then $K_{n}\subset K_{n-1}$.
\end{itemize}

%

\noindent
Suppose that for some $n\geqslant0$ we have constructed $\lambda_{n-1}%
,g_{n-1},K_{n-1},\varepsilon_{n-1},$ and $x_{n-1}$ with the applicable
properties. If $\lambda_{n-1}=0$, choose $x_{n}\in K_{n-1}$ such that
$f(x_{n})<\inf f(K_{n-1})+2^{-n}$. By Corollary \ref{may25c3}, $\inf g_{n}$
exists, as does $\varepsilon_{n}$ with $0<\varepsilon_{n}<\frac{1}{2}%
\min\left\{  \varepsilon,f(x_{n})\right\}  $ and $f(x_{n})-\varepsilon_{n}%
\neq\inf g_{n}$, such that if $f(x_{n})-\varepsilon_{n}>\inf g_{n}$, then
$\mathsf{L}(g_{n},f(x_{n})-\varepsilon_{n})\ $is compact. Either
$f(x_{n})-\varepsilon_{n}<\inf g_{n}$, in which case $\mathsf{L}(g_{n}%
,f(x_{n})-\varepsilon_{n})$ is empty and we set $\lambda_{n}=1$ and
$K_{n}=\left\{  x_{n}\right\}  $; or else $f(x_{n})-\varepsilon_{n}>\inf
g_{n}$, in which case we set $\lambda_{n}=0$ and $K_{n}=\mathsf{L}%
(g_{n},f(x_{n})-\varepsilon_{n})$. This disposes of the alternative
$\lambda_{n-1}=0$. If, on the other hand, $\lambda_{n-1}=1$, we set
$\lambda_{n}=1$, $g_{n}=g_{n-1}$, $K_{n}=\left\{  x_{n-1}\right\}  $,
$\varepsilon_{n}=\varepsilon_{n-1}$, and $x_{n}=x_{n-1}$. This completes the
inductive construction itself. Clearly, (a) and (b) are satisfied. It is
straightforward to prove (c), so it remains to prove (d). To that end, if
$\lambda_{n}=0$, then for each $x\in K_{n}$ we have
\begin{align*}
f\left(  x\right)  +c\rho\left(  x,x_{n-1}\right)   &  \leqslant f\left(
x\right)  +c\rho\left(  x,x_{n}\right)  +c\rho\left(  x_{n},x_{n-1}\right) \\
&  \leqslant f\left(  x_{n}\right)  -\varepsilon_{n}+c\rho\left(
x_{n},x_{n-1}\right) \\
&  \leqslant f\left(  x_{n-1}\right)  -\varepsilon_{n}-\varepsilon
_{n-1}\ \ \text{(as }x_{n}\in K_{n-1}\text{)}\\
&  <f\left(  x_{n-1}\right)  -\varepsilon_{n-1}%
\end{align*}
and therefore $x\in K_{n-1}$; whence $K_{n}\subset K_{n-1}$. If $\lambda
_{n}=1-\lambda_{n-1}$, then $K_{n}=\left\{  x_{n}\right\}  \subset K_{n-1}$,
by (c)~and (b); and if $\lambda_{n-1}=1$, then $K_{n}=K_{n-1}$, by (c). This
completes the proof of (d).

If $n\geqslant1$ and $\lambda_{n}=0$, then for each $x\in K_{n}$,%
\[
\inf f(K_{n})+c\rho\left(  x,x_{n}\right)  \leqslant f\left(  x\right)
+c\rho\left(  x,x_{n}\right)  \leqslant f\left(  x_{n}\right)  -\varepsilon
_{n},
\]
so, noting that $\inf f(K_{n})\geqslant\inf f(K_{n-1})$, we have
\[
\rho\left(  x,x_{n}\right)  \leqslant c^{-1}(f\left(  x_{n}\right)  -\inf
f(K_{n})-\varepsilon_{n})<c^{-1}(f\left(  x_{n}\right)  -\inf f(K_{n-1}%
))<\frac{1}{2^{n}c}.
\]
It follows from this that $\mathsf{diam}(K_{n})<1/2^{n-1}c$ for all $n$. Hence
$\left(  K_{n}\right)  _{n\geqslant0}$ is a descending sequence of compact
sets whose diameters tend to $0$. By Proposition \ref{may04p1}, $\left(
x_{n}\right)  _{n\geqslant1}$ converges to the unique point$\ v~$of $%
{\textstyle\bigcap_{n\geqslant0}}
K_{n}$. Since $v\in K_{0}$, statement (\ref{1}) clearly holds. On the other
hand, given $x\in X$ with $x\neq v$, we have either (\ref{2}) or%
\begin{equation}
f\left(  x\right)  +c\rho\left(  x,v\right)  <f\left(  v\right)
-\frac{\varepsilon}{2}. \label{06a}%
\end{equation}
To complete the proof, it will suffice to rule out the latter alternative.
Assuming (\ref{06a}), if $\lambda_{n}=1-\lambda_{n-1}$ we have $\mathsf{L}%
(g_{n},f(x_{n})-\varepsilon_{n})=\varnothing$ and $x_{m}=x_{n}$ for all
$m\geqslant n$. Hence $v=x_{n}$ and therefore%
\[
f\left(  x\right)  +c\rho\left(  x,v\right)  =f\left(  x\right)  +c\rho\left(
x,x_{n}\right)  \geqslant f\left(  x_{n}\right)  -\varepsilon_{n}\geqslant
f\left(  v\right)  -\frac{\varepsilon}{2},
\]
contradicting (\ref{06a}). Thus, recalling that $\lambda_{0}=0$, for all $n$
we must have $\lambda_{n}=0$ and $K_{n}$ equal to the compact set$\ \mathsf{L}%
(g_{n},f(x_{n})-\varepsilon_{n})$. By (\ref{06a}) and the continuity of $f $
and $\rho$, for all sufficiently large $n$ we have%
\[
f(x)+c\rho(x,x_{n})+\varepsilon_{n}<f(x)+c\rho(x,x_{n})+\frac{\varepsilon}%
{2}<f(x_{n});
\]
whence $x\in K_{n}$ and therefore%
\[
\rho(x_{n},x)\leqslant\mathsf{diam}(K_{n})<\frac{1}{2^{n-1}c}\rightarrow
0\text{ as }n\rightarrow\infty\,.
\]
Thus $x=\lim_{n\rightarrow\infty}x_{n+1}=v$, a contradiction which ensures
that the alternative (\ref{06a}) is ruled out.
\end{proof}

%

\medskip

We now have a constructive version of a well-known classical consequence of
Theorem \ref{Ekeland} that provides us with approximate minima.

\begin{corollary}
\label{Ekecor}Let $f$ be a uniformly continuous mapping on a locally compact
metric space $X$ such that $\inf f$ exists. Let $\varepsilon>0$ and let
$x_{0}$ be a point of $X$ such that $f(x_{0})<\inf_{x\in X}f+\varepsilon$.
Then for each $\lambda>0$ there exists $v\in X$ such that $f(v)\leqslant
f(x_{0}),\ \rho\left(  v,x_{0}\right)  <\sqrt{\varepsilon}$, and%
\begin{equation}
f\left(  v\right)  <f(x)+\sqrt{\varepsilon}\rho\left(  v,x\right)
+\varepsilon\ \label{33}%
\end{equation}
whenever $x\in X$ and $x\neq v$.
\end{corollary}

\begin{proof}
Taking $\alpha=\sqrt{\varepsilon}$ in Theorem \ref{Ekeland}, construct $v\in
X$ such that (\ref{1}) and (\ref{33}) hold. Then%
\[
f(v)+\sqrt{\varepsilon}\rho\left(  v,x_{0}\right)  \leqslant f\left(
x_{0}\right)  <\inf f+\varepsilon\leqslant f(v)+\varepsilon,
\]
so $f(v)\leqslant f(x_{0})$ and $\rho(v,x_{0})<\sqrt{\varepsilon}$.
\end{proof}

%

\bigskip
%

\noindent
\textbf{Author's address: \ }Department of Mathematics \& Statistics,
University of Canterbury, Christchurch 8140, New Zealand%

\noindent
\textbf{Author's email: \ \ \ \ \ }\texttt{dugbridges@gmail.com}

\end{document}